\documentclass[12pt,oneside]{amsart}
\usepackage{amsmath,amsthm,amssymb,epic,eepic,bm}
\usepackage{eqlist,eqparbox}
\usepackage{comment}

\input diagxy

\title[Operators induced by fuzzy relations]
{\bf Operators induced by fuzzy relations}
\author {Michal Botur}

\thanks{Author gratefully acknowledge  the support by ESF Project CZ.1.07/2.3.00/20.0051 Algebraic methods in Quantum Logic of the Masaryk University.} 
\address{Palack\' y University Olomouc, Faculty of Sciences,  17.listopadu 1192/12, Olomouc 771 46, Czech Republic}
\email{michal.botur@upol.cz}

\begin{document}

\newtheorem{defin}{Definition}
\newtheorem{prop}{Proposition}
\newtheorem{theorem}{Theorem}
\newtheorem{lemma}{Lemma}
\newtheorem{cor}{Corollary}
\newtheorem{remark}{Remark}
\newtheorem{cl}{Claim}
\newtheorem{example}{Example}
\newcommand{\zl}{\mbox{$[\hspace{-1.8pt} [$}}
\newcommand{\zr}{\mbox{$]\hspace{-2pt} ]$}}
\newcommand{\s}{\mathrm{Spec_M}\,}

\begin{abstract}
Theory of operators generated by binary fuzzy relations is highly increasing for its nature and applicability. The main goal of the paper is to present several representation theorems for operators induced by fuzzy relations (for example closure operators used in formal concept analysis, monadic operators or tense operators). Consequently we establish  algebraic models with their semantics which are usable in the non-classical logic research and in the computer science research. The obtained results are applied in the theory of Pavelka's algebras.
\end{abstract}

\maketitle

\section{Introduction}
The fuzzy logic brings theoretic background to analyse non boolean inputs. The classical mathematical logic loses its applicability in the case when there is no natural way to describe the analysed data, properties or relations using 0/1 (false/true). A key idea is to extend the two element truth scale to a richer one. Logics obtained by the extensions are called fuzzy logics. There exist several theoretic concepts introducing a logic with vagueness or uncertainty. In this paper we use the concept of a \emph{commutative bounded integral residuated lattice}
$$\mathbf A = (A;\vee,\wedge,\cdot,\rightarrow,0,1)$$
which we will call simply a \emph{residuated lattice}. Thus,
\begin{itemize}
\item[i)] $(A;\vee,\wedge,0,1)$ is a bounded lattice,
\item[ii)] $(A;\cdot,1)$ is a commutative monoid,
\item[iii)] the adjointness property holds, i.e.,
$$x\cdot y\leq z \,\, \mbox{ if, and only if, }\,\, x\leq y\rightarrow z.$$
\end{itemize}
We understand this structure as a general algebraic model of a truth scale. The elements 0 and 1 model the strict false and the strict true. Other elements represent the fuzzy truth degrees. The connective of logical conjunction is modelled by the operation $\cdot$, the implication by $\rightarrow$.

Important models are residuated lattices induced by triangular norms. By a \emph{triangular norm} we mean a binary operation $\cdot$ defined on the real interval $[0,1]$ which is commutative, associative and left-continuous (in usual sense) and monotone. Then the residuum operation is given by 
$$x\rightarrow y=\max\{a\mid a\cdot x\leq y\}.$$
Used order is the standard one.

\begin{example}
The \L ukasiewicz triangular norm is defined by $x\cdot y =\min\{1-x-y,0\}$. Its residual operation is $x\rightarrow y =\min\{1-x+y,1\}$
\end{example}

\begin{example}
The G\" odel triangular norm is defined by $x\cdot y = \min\{x,y\}$ and its residual operation is $x\rightarrow y = y$ if $y\leq x$ and $x\rightarrow y = 1$ if $x<y$.
\end{example}

\begin{example}
Product triangular norm is just the standard product of real numbers and its residual operation is $x\rightarrow y =\min \{x/y,1\}$ if $y\not = 0$ and $x\rightarrow 0 = 0$ if $x\not = 0$ and $0\rightarrow 0= 1$. 
\end{example}

It is well known that all continuous triangular norms are decomposable into \L ukasiewicz,  G\" odel and product ones (for the details we refer to \cite{Haj}). Moreover, the H\' ajek's basic logic is generated exactly by the residuated lattices induced by the continuous triangular norms  \cite{CEGT}. Altogether, triangular norms, the H\' ajek's basic logic, or its special subclasses are the most applicable classes of fuzzy logics in the computer science. 

Algebraic models of the H\' ajek's basic logic are the, so called, \emph{BL-algebras} which are just residuated lattices satisfying the divisibility law $$x\cdot (x\rightarrow y)= x\wedge y$$ and the prelinearity law $$(x\rightarrow y)\vee(y\rightarrow x)=1.$$ BL-algebras satisfying the double negation law $\neg\neg x = x$ where $$\neg x := x\rightarrow 0$$ are called MV-algebras. The class of MV-algebras is induced exactly by the \L ukasiewicz triangular norm.

Recall that MV-algebras are usually defined as algebras of type $\mathbf A=(A;\oplus,\neg,0)$ such that
\begin{itemize}
\item[(MV1)] $(A;\oplus, 0)$ is a commutative monoid,
\item[(MV2)] the double negation $\neg\neg x=x$ holds,
\item[(MV3)] the \L ukasiewicz axiom $\neg(\neg x\oplus y)\oplus y=\neg(\neg y\oplus x)\oplus x$ holds. 
\end{itemize}
However, both presented definitions are equivalent. More precisely, if $\mathbf A=(A;\vee,\wedge,\cdot,\rightarrow,0,1)$ is residuated lattice satisfying divisibility, prelinearity and double negation law then an algebra $(A;\oplus,\neg,0)$ where
$x\oplus y := \neg(\neg x\cdot \neg y)$
is an MV-algebra. Conversely, let us have an MV-algebra $\mathbf A=(A;\oplus,\neg,0)$ then an algebra $(A;\vee,\wedge,\cdot,\rightarrow,0,1)$, where $$x\vee y:= \neg(\neg x\oplus y)\oplus y,$$ $$ x\wedge y := \neg (\neg x\vee \neg y),$$ $$ x\cdot y:= \neg(\neg x\oplus \neg y),$$ $$ x\rightarrow y= \neg x\oplus y,$$ $$ 1:=\neg 0$$ is a residuated lattice satisfying divisibility, prelinearity and double negation law. We remark that the induced order can be described by a stipulation $$x\leq y\,\,\mbox{ if, and only if, }\,\, \neg x\oplus y = 1.$$

We remind that the results contained in this paper are formulated for more general structures; thus all the results are fully applicable in the theory of BL-algebras.

The following lemma \cite{Haj} describes preserving or reversing of infima and suprema.

\begin{lemma}\label{rez}
Let us have a resiuated lattice $\mathbf A=(A;\vee,\wedge,\cdot,\rightarrow,0,1)$, let $M\subseteq A$ be an arbitrary set and let $x\in A$ be an arbitrary element.
\begin{itemize}
\item[i)] If the supremum $\bigvee M$ exists then also the supremum $\bigvee \{x\cdot m\mid m\in M\}$ exists and $$x\cdot\big(\bigvee M\big) = \bigvee \{x\cdot m\mid m\in M\}.$$ 
\item[ii)] If the supremum $\bigvee M$ exists then also the infimum $\bigwedge \{m\rightarrow x\mid m\in M\}$ exists and $$\big(\bigvee M\big)\rightarrow x = \bigwedge \{m\rightarrow x\mid m\in M\}.$$ 
\item[iii)] If the infimum $\bigwedge M$ exists then also the infimum $\bigwedge \{x\rightarrow m\mid m\in M\}$ exists and $$x\rightarrow\big(\bigwedge M\big) = \bigwedge \{x\rightarrow m\mid m\in M\}.$$ 
\end{itemize}
\end{lemma}

\section{Galois connections}

Let us have ordered sets $(A;\leq)$ and $(B;\leq)$. Then a couple of monotone mappings 
$$\bfig
\morphism(0,0)|a|/{@{>}@/^1em/}/<500,0>[A`B;f]
\morphism(0,0)|b|/{@{<-}@/_1em/}/<500,0>[A`B;g]
\efig$$
forms a \emph{Galois connection} if they satisfy the adjointness property
$$x\leq f(y)\,\, \mbox{ if, and only if, }\,\, g(x)\leq y\eqno{(1)}$$
for all $y\in A$ and $x\in B$. A couple of antitone mappings
$$\bfig
\morphism(0,0)|a|/{@{>}@/^1em/}/<500,0>[A`B;d]
\morphism(0,0)|b|/{@{<-}@/_1em/}/<500,0>[A`B;h]
\efig$$
forms a \emph{reversed (or contravariant) Galois connection} if it satisfies
$$x\leq d(y) \mbox{ if and only if } y \leq h(x)$$
for all $y\in A$ and all $x\in B$. In this paper we denote the composition of mappings $fg\colon A\longrightarrow C$ for given mappings $f\colon A\longrightarrow B$ and $g\colon B\longrightarrow C$. We recall the well known lemma characterizing Galois connections.

\begin{lemma}\label{ll1}
Let us have ordered sets $(A;\leq)$ and $(B;\leq)$ and monotone mappings $f\colon A\longrightarrow B$, $g\colon B\longrightarrow A$. The following properties are equivalent:
\begin{itemize}
\item[i)] $f$ and $g$ form a Galois connection,
\item[ii)] $gf(y)\leq y$ and $x\leq fg(x)$ hold for any $y\in A$ and $x\in B.$
\end{itemize}
\end{lemma}
\begin{proof}
{\it i) $\Rightarrow$ ii).} If $f$ and $g$ form a Galois connection then from $f(y)\leq f(y)$ and (1) we have $gf(y)\leq y$ for all $y\in A$. Analogously, $g(x)\leq g(x)$ gives $x\leq fg(x)$ for all $x\in B$.

{\it ii) $\Rightarrow$ i).} If $x\leq f(y)$ fore some $y\in A$ and $x\in B$ then, using monotonicity and assumed inequalities, we obtain $g(x)\leq gf(y)\leq y$. Conversely, $g(x)\leq y$ yields $x\leq fg(x)\leq f(y)$.
\end{proof}

We remark that a \emph{closure operator} on an ordered set $(A;\leq)$ is a monotone mapping $C\colon A\longrightarrow A$ satisfying $x\leq C(x)$ and $CC(x)=C(x)$ for all $x\in A$. Dually, an \emph{interior operator} on an ordered set $(A;\leq)$ is  a monotone mapping $I\colon A\longrightarrow A$ satisfying $I(x)\leq x$ and $II(x)=I(x)$ for all $x\in A$.

\begin{cor}\label{c1}
If the mappings $f$ and $g$ form a Galois connection between ordered sets $(A;\leq)$ and $(B;\leq)$ then $gf\colon A\longrightarrow A$ is a interior operator on $(A;\leq)$ and interiors are just elements $g(x)$ for any $x\in B$.

The operator $fg\colon B\longrightarrow B$ is a closure operator and closed elements are just elements $f(y)$ for any $y\in A$.      
\end{cor}
\begin{proof}
In Lemma \ref{ll1} we have proved the inequality $gf(y)\leq y$ for all $y\in A$. We can also deduce $fgf (y)\leq f(y)$. Lemma \ref{ll1} further stated an inequality $x\leq fg (x)$ and thus $f(y)\leq fgf(y)$ holds. Together $fgf(y)=f(y)$ holds which prove that $gfgf(y)=gf(y)$ and $fgfg(y)=fg(y)$. Hence, $gf$ is an interior operator and analogously $fg$ is an interior operator. Moreover, it is proved that $f(y)$ are interiors. Conversely, if $y\in A$ is an interior then $gf(y)= y$ and if we denote $f(y)=x\in B$ then $y=g(x)$. Hence interiors are just elements $g(x)$ where $x\in B$. The rest of the proof is analogous.
\end{proof}

Because reversed Galois connections between ordered sets $(A;\leq)$ and $(B;\leq)$ are just Galois connections between $(A;\leq)$ and $(B;\leq^{-1})$, where $\leq^{-1}$ denotes the inverse order, we can immediately state the following lemma and corollary.

\begin{lemma}\label{ll2}
Let us have ordered sets $(A;\leq)$ and $(B;\leq)$ and antitone mappings $d\colon A\longrightarrow B$, $h\colon B\longrightarrow A$. The following properties are equivalent:
\begin{itemize}
\item[i)] $d$ and $h$ form a reversed Galois connection,
\item[ii)] $y\leq hd(y)$ and $x\leq dh(x)$ hold for all $y\in A$ and $x\in B.$
\end{itemize}
\end{lemma}

\begin{cor}\label{c2}
If the mappings $d$ and $h$ form a Galois connection between ordered sets $(A;\leq)$ and $(B;\leq)$ then $hd\colon A\longrightarrow A$ is a closure operator on $(A;\leq)$ and closed elements are just elements $h(x)$ for any $x\in B$.

The operator $hd\colon B\longrightarrow B$ is a closure operator and closed elements are just elements $d(y)$ for any $y\in A$.   
\end{cor}

\begin{defin}
We say that a mapping $f\colon A\longrightarrow B$, where $(A;\leq)$ and $(B;\leq)$ are ordered sets, is \emph{infima preserving} (or \emph{suprema preserving}) if, for any $M\subseteq A$ such that $\bigwedge M$ exists, also $\bigwedge f(M)$ exists and an equality $f(\bigwedge M)=\bigwedge f(M)$ holds (if $\bigvee M$ exists then also $\bigvee f(M)$ exists and an equality $f(\bigvee M)=\bigvee f(M)$ holds).

We say that a mapping $f\colon A\longrightarrow B$ is \emph{infima reversing} (or \emph{suprema reversing}) if, for any $M\subseteq A$ such that $\bigwedge M$ exists, also $\bigvee f(M)$ exists and an equality $f(\bigwedge M)=\bigvee f(M)$ holds (if $\bigvee M$ exists then also $\bigwedge f(M)$ exists and an equality $f(\bigvee M)=\bigwedge f(M)$ holds).
\end{defin} 

The following well known theorem states necessary and sufficient conditions for an existence of an opposite mapping forming a Galois connection if the first one is given and it shows a way how to derive the second (opposite) mapping from the given one.

\begin{theorem}\label{rep}
Let us have two complete lattices $(A;\vee,\wedge,0,1)$ and $(B;\vee,\wedge,0,1)$.
\begin{itemize}
\item[i)] If $f\colon A\longrightarrow B$ is monotone function then there exists $g\colon B\longrightarrow A$ such that $f$ and $g$ form a Galois connection if, and only if, $f$ is an infima preserving;  then $g(x)=\bigwedge \{a\in A\mid x \leq f(a) \}.$ 
\item[ii)] If $g\colon B\longrightarrow A$ is monotone function then there exists $f\colon A\longrightarrow B$ such that $f$ and $g$ forms a Galois connection if, and only if, $g$ is a suprema preserving; then $f(x)=\bigvee \{b\in B\mid g(b) \leq x \}$. 
\item[iii)] If $d\colon A\longrightarrow B$ is monotone function then there exists $h\colon B\longrightarrow A$ such that $d$ and $h$ forms reversing Galois connection if, and only if, $d$ is a suprema reversing; then $h(x)=\bigvee \{a\in A\mid x \leq d(a) \}$.
\end{itemize}
\end{theorem}   

Direct corollary of the previous theorem is uniqueness of the opposite mappings forming (reversing) Galois connection with given one. We are going to show important examples of Galois connections used in substructural logic and computer science.

\subsection{Universal and Existential Quantifiers} Universal or existential quantifiers are modelled by the, so called, monadic operators. Monadic operator $\exists$ on a Boolean algebra $\mathbf B$ can be defined as a closure operator satisfying $\exists \neg \exists x =\neg \exists x$ (see \cite{Hal}). Dual operator $\forall$ is derived as $\forall x= \neg \exists \neg x$ and it is an interior operator satisfying $\forall\neg\forall x= \neg\forall x$. It can be easily verified that inequality $$\exists\forall x\leq x\leq \forall\exists x$$ holds. Using Lemma \ref{ll1} we can prove that the operators $\forall$ and $\exists$ form Galois connection on $\mathbf B$ (it means Galois connection between $\mathbf B$ and $\mathbf B$). Moreover, a monadic operators on a boolean algebra $\mathbf B$ is just an interior operator $\forall\colon B\longrightarrow B$ which forms a Galois connection with $\neg \forall\neg$ on $\mathbf B$.

Theory of monadic operators was generalized for MV-algebras \cite{DiNG, Geo1}. If $\mathbf A = (A;\oplus,\neg,0)$ is an MV-algebra then a monadic operator $\exists \colon A\longrightarrow A$ was defined by the inequalities
\begin{itemize}
\item[($\exists$1)] $x\leq \exists x$,
\item[($\exists$2)] $\exists (x\vee y)=\exists x \vee \exists y$,
\item[($\exists$3)] $\exists \neg\exists x= \neg \exists x$,
\item[($\exists$4)] $\exists (\exists x\oplus \exists y)=\exists x\oplus \exists y$,
\item[($\exists$5)] $\exists (x\oplus x)=\exists x\oplus \exists x$
\item[($\exists$6)] $\exists (x\cdot x)= \exists x \oplus \exists x$.
\end{itemize}
This definition still guarantee that any monadic operator $\exists$ with $\neg\exists\neg$ form a Galois connection. In this paper we will try to show (inter alia) that this definition is stronger that it could be. 

There exist several generalizations of monadic operators for another structures \cite{RaSa, RaSa1} but monadic operators on structures with no double negation law lost natural Galois connections and obtained results are weaker. 

\subsection{Tense Operators} To obtain the, so-called, {\it tense logic} from the classical logic the propositional calculus is enriched by new unary operators $G$ and $H$ (and new derived operators $F:=\neg G\neg$ and $P:=\neg H \neg,$ where $\neg$ denotes the classical negation) which are called {\it tense operators}. The operator $G$ usually express the quantifier `it will still be the case that' and $H$ express `it has always been the case that'. Hence, $F$ and $P$ are in fact tense existential quantifiers.

The couple $(T,\rho)$ where $T$ is a non-void set and $\rho$ is a binary relation on $T$  is called a {\it time frame}. For a given logical formula $\phi$ of our propositional logic and for $t\in T$ we say that $G(\phi(t))$ is valid if $\phi(s)$ is valid for any $s\in T$ with $t\rho s.$ Analogously, $H(\phi (t))$ is valid if $\phi(s)$ is valid for any $s\in T$ with $s\rho t.$ Thus $F(\phi(t))$ is valid if there exists $s\in T$ such that $t\rho s$ and $\phi (s)$ is valid and analogously $P(\phi (t))$ is valid if there exists $s\in T$ such that $s \rho t$ and $\phi(s)$ is valid.

Study of tense operators has originated in 1980's \cite{2}. Recall that for a classical propositional calculus represented by the means of  Boolean algebra $\mathbf B=(B;\vee,\wedge,\neg,0,1)$ tense operators were axiomatized \cite{2} by the following axioms:
\begin{itemize}
\item[(B1)] $G(1)=1,$ $H(1)=1,$
\item[(B2)] $G(x\wedge y)= G(x)\wedge G(y),$ $H(x\wedge y)=H(x)\wedge H(y),$
\item[(B3)] $\neg G\neg H (x)\leq x,$ and $\neg H\neg G (x)\leq x.$
\end{itemize}

Lemma \ref{ll1} and Axiom (B3) prove that operators $H$ and $\neg G \neg$ form a Galois connection on $\mathbf B$. Well known representation theorem states that every Boolean algebra with tense operators $G$ and $H$ can be embedded into a Boolean algebra $2^T$ where the operators $G$ and $H$ are induced by a binary relation $\rho\subseteq T^2$ by the stipulations
$$G(x)(i)=\bigwedge_{i\rho j} x(j)\mbox{ and } H(x)(i)=\bigwedge_{j\rho i} x(j)\eqno{(1)}$$ for all $i\in T$.
The idea of tense operators was used for more general constructions in MV-algebras. Tense MV-algebras were introduced by D. Diagonescu and G.Georgescu in \cite{7} as following 

If $\mathbf A=(A;\oplus,\neg,0)$ is an MV-algebra then $(\mathbf A,G,H)$ is a \emph{tense MV-algebra} and $G$ and $H$ are \emph{tense operators} if $G$ and $H$ are unary operators on $A$ satisfying:
\begin{itemize}
\item[(T1)] $G(1)=H(1)=1,$
%\item[()] $G(x\wedge y)=G(x)\wedge G(y),$ $H(x\wedge y)=H(x)\wedge H(y),$
\item[(T2)] $G(x)\cdot G(y)\leq G (x\cdot y),$ $H(x)\cdot H(y)\leq H(x\cdot y),$
\item[(T3)] $G(x)\oplus G(y)\leq G (x\oplus y),$ $H(x)\oplus H(y)\leq H(x\oplus y),$
\item[(T4)] $G(x)\cdot G(x)=G(x\cdot x),$ $H(x)\cdot H(x)=H(x\cdot x)$,
\item[(T5)] $G(x)\oplus G(x)=G(x\oplus x),$ $H(x)\oplus H(x)=H(x\oplus x)$,
\item[(T6)] $\neg G\neg H (x)\leq x,$ $\neg H\neg G (x)\leq x.$ 
\end{itemize}
The representation theorem for the semisimple tense MV-algebras was proved by the author and J. Paseka \cite{BoPa}. Thus any tense MV-algebra defined on a semisimple MV-agebra is embeddable into a tense MV-algebra defined on $[0,1]^T$ where the operators $G$ and $H$ possesses a binary relation $\rho\subseteq T^2$ such that (1) hold. Moreover, $\forall$ is a monadic operator if, and only if, $\forall = G= H$ are tense operators induced by a relation equivalence.

Thus both monadic and tense operators are just special cases of Galois connections induced by a (boolean) binary relation. These operators work  well in the \L ukasiewicz logic (MV-algebras) thanks to the double negation law. Definition of monadic (resp. tense) MV-algebras guarantees that a monadic (resp. tense) MV-algebra restricted to its boolean elements is a monadic (resp. tense) boolean sublagebra. This fact is an implicit consequence of the inducibility of these operators by {\bf boolean} binary relations. We are going to show that a generalization of the above presented constructions for fuzzy relations brings a richer theory with out any lost of naturality.  

The theory of tense operators on algebraic models of logics was recently studied for example, for the intuitionistic logic (corresponding to Heyting algebras) in \cite{3}, and algebras of logic of quantum mechanics  \cite{4,5}, the so called basic algebras \cite{1}, and other interesting algebras \cite{8,9,10}.

\subsection{Formal Concept analysis} Formal concept analysis has interesting applications in several fields of computer science (for example data mining, machine learning, artificial intelligence etc.). The notion was introduced by R.~Wille in 1984 and the basic idea comes from philosophy and linguistic theory. By a \emph{formal context} we mean a triple $(G,M,I)$ where $G$ is a set representing a set of objects, $M$ is a set of attributes and $I\subseteq G\times M$ is a binary relation representing which objects possess which attributes.

Any formal context induce a pair of operators $d\colon \mathbf 2^G\longrightarrow \mathbf 2^M$
 and $h\colon \mathbf{2}^M\longrightarrow \mathbf 2^G$ (where $\mathbf 2^M$ and $\mathbf 2^G$ denote the power sets of $M$ and $G$) defined by
 $$d(X)=\{m\in M\mid xIm \mbox{ for all }x\in X\}\mbox{ for all }X\subseteq M,$$
 $$h(Y)=\{g\in G\mid gIy \mbox{ for all }y\in Y\}\mbox{ for all }Y\subseteq G.$$
Operators $d$ and $h$ form a reversing Galois connection between $\mathbf 2^M$ and $\mathbf 2^G$. For the details we refer readers to \cite{Wil}. Thus $dh$ and $hd$ are closure operators on $\mathbf 2^M$ and $\mathbf 2^G$. Concepts are couples $(X,Y)$ where $X\subseteq G$, $Y\subseteq H$,  $d(X)=Y$ and $h(Y)=X$. According to Lemma \ref{ll2}, concepts are just pairs of closed sets in the form $(dh(X),h(X))$ where $X\subseteq G$.

Generalization of formal concept analysis for fuzzy logics was introduced by R. B\v elohl\' avek \cite{Bel1, Bel}. For a complete residuated lattice $\mathbf A = (A;\vee,\wedge,\cdot,\rightarrow,0,1)$, a \emph{fuzzy formal context} is a triple $(G,M,I)$ where (in contrast with the  original formal context) $I$ is a fuzzy relation, i.e., a mapping $I\colon G\times M\longrightarrow A$. 

Analogously to the boolean case, any fuzzy formal context induce the operators $d\colon \mathbf A^G\longrightarrow \mathbf A^M$ and $h\colon \mathbf{A}^M\longrightarrow \mathbf A^G$ defined by
$$d(x)(j)=\bigwedge_{i\in M}(x(i)\rightarrow I(i,j))\mbox{ for all }j\in G,$$
$$h(x)(j)=\bigwedge_{i\in G}(x(i)\rightarrow I(j,i))\mbox{ for all }j\in M.$$
Operators $d$ and $h$ forms reversing Galois connection between $\mathbf A^M$ and $\mathbf A^G$ as well. Thus $dh$ and $hd$ are closure operators on $\mathbf A^M$ and $\mathbf A^G$. Concepts are couples $(x,y)$ where $x\in A^M$, $y\in A^G$, $d(x)=y$, and $h(y)=x$ (or equivalently $(dh(x),h(x))$ for any $x\in A^G$).

Another goal of this paper is to give a representation theorem for reversing Galois connections induced by (fuzzy) formal concepts. Thus we will be able to decide whether given reversing Galois connections are induced by some fuzzy formal concept, or not.

\section{Fuzzy binary relations}

Let us have a residuated lattice $\mathbf A=(A;\vee,\wedge,\cdot,\rightarrow,0,1)$. By an \emph{$\mathbf A$ fuzzy binary relation} (or, briefly, a \emph{fuzzy relation}) between sets $I$ and $J$ we mean any mapping 
$$R \colon I\times J\longrightarrow A.$$
This relation can be interpreted `an element $i$ is in the relation $R$ with a element $j$ in a degree $R (i,j)$'. Classical relations are then modelled by $\{0,1\}$-valued mappings.

\begin{defin}
Let us have a residuated lattice $\mathbf A = (A;\vee,\wedge,\cdot,\rightarrow,0,1)$, arbitrary set $I$ and a $\mathbf A$-relation $R\colon I\times I\longrightarrow A$. Then
\begin{itemize}
\item[i)] $R$ is \emph{reflexive} if $R(i,i)=1$ for all $i\in I$,
\item[ii)] $R$ is \emph{symmetric} if $R(i,j)=R(j,i)$ for all $i,j\in I$,
\item[iii)] $R$ is \emph{transitive} if $R(i,j)\cdot R(j,k)\leq R(i,k)$ for all $i,j,k\in I$.
\end{itemize}
\end{defin}
 
\subsection{Operators induced by fuzzy binary relations} Several operators induced by fuzzy relation were mentioned in the introduction. The purpose of the following definition is to unite the notation.

\begin{defin}
Let us have a complete residuated lattice $\mathbf A = (A;\vee,\wedge,\cdot,\rightarrow,0,1)$, arbitrary sets $I$ and $J$ and an $\mathbf A$-relation $R\colon I\times J\longrightarrow A$. We introduce the operators $\phi_R\colon A^I\longrightarrow A^J$, $\rho_R\colon A^J\longrightarrow A^I$, $\delta_R\colon A^I\longrightarrow A^J$ and $\epsilon_R\colon A^J\longrightarrow A^I$ by the stipulations 
$$\phi_R(x)(j)=\bigwedge_{i\in I} (R(i,j)\rightarrow x(i)) \mbox{ for all } j\in J,\eqno{(\phi_R)}$$
$$\rho_R(x)(i)=\bigvee_{j\in J} (R(i,j)\cdot x(j)) \mbox{ for all } i\in I,\eqno{(\rho_R)}$$
$$\delta_R(x)(j)=\bigwedge_{i\in I} (x(i)\rightarrow R(i,j)) \mbox{ for all } j\in J,\eqno{(\delta_R)}$$
$$\epsilon_R(x)(i)=\bigwedge_{j\in I} (x(j)\rightarrow R(i,j)) \mbox{ for all } i\in I.\eqno{(\epsilon_R)}$$
\end{defin}

Let us have arbitrary sets $A$ and $I$. Then an element $d\in A^I$ is called \emph{diagonal} if it satisfies $d(i)=d(j)$ for any $i,j\in I$. If $d\in A$ then we denote a diagonal element $d^I\in A^I$ by $d^I(i)=d$ for any $i\in I$.

\begin{defin}
Let us have a complete residuated lattice $\mathbf A=(A;\vee,\wedge,\cdot,\rightarrow,0,1)$ and let $I$ and $J$ be arbitrary sets. Then 
\begin{itemize}
\item[($\phi$)] a mapping $\phi: A^I\longrightarrow  A^J$ is called a \emph{$\phi$-type mapping} if it is infima preserving and if the equality $$d^J\rightarrow \phi(x)=\phi(d^I\rightarrow x)$$ holds for all $d\in A$ and $x\in A^I$, 
\item[($\rho$)] a mapping $\rho: A^J\longrightarrow  A^I$ is called a \emph{$\rho$-type mapping} if it suprema preserving and if the equality $$d^I\cdot \rho(x)=\rho(d^J\cdot x)$$ holds for all $d\in A$ and $x\in A^J$, 
\item[($\delta$)] a mapping $\delta: A^I\longrightarrow  A^J$ is called a \emph{$\delta$-type mapping} if it is suprema reversing and if the equality $$d^J\rightarrow \delta(x)=\delta(d^I\cdot x)$$ holds for all $d\in A$ and $x\in A^I$. 
\end{itemize}
\end{defin}

Since $0\in A^I$ and $0\in A^J$ are diagonal elements we obtain
$$\phi (1)=\phi(0\rightarrow x)=0\rightarrow \phi (x)=1,$$
$$\rho (0)=\rho (0\cdot x)=0\cdot \rho (x) = 0,$$
$$\delta (0)=\delta (0\cdot x)=0\rightarrow \delta (x)=1.$$
The defined mappings types are transferable through Galois connection. More precisely:

\begin{lemma}\label{l1}
Let us have a complete residuated lattice $\mathbf A=(A;\vee,\wedge,\cdot,\rightarrow,0,1)$ and let $I$ and $J$ be arbitrary sets. 
\begin{itemize}
\item[i)] Let us have mappings $\phi\colon A^I\longrightarrow A^J$ and $\rho\colon A^J\longrightarrow A^I$ such that $\phi$ and $\rho$ form a Galois connection between $A^I$ and $A^J$. Then $\phi$ is a $\phi$-type mapping if and only if $\rho$ is a $\rho$-type mapping.
\item[ii)] Let us have mappings $\delta\colon A^I\longrightarrow A^J$ and $\epsilon\colon A^J\longrightarrow A^I$ such that $\delta$ and $\epsilon$ form a reversed Galois connection between $A^I$ and $A^J$. Then $\delta$ is a $\delta$-type mapping if and only if $\epsilon$ is a $\delta$-type mapping.
\end{itemize}
\end{lemma}
\begin{proof}
Let the mappings $\phi$ and $\rho$ form a Galois connection between $A^I$ and $A^J$. Theorem \ref{rep} shows that $\phi$ is infima preserving and $\rho$ is suprema preserving. If  $\phi$ is a $\phi$-type mapping then for any $d\in A$, $x\in A^I$ and $y\in A^J$ we have
\begin{eqnarray*}
d^I\cdot \rho(y)\leq x &\text{if and only if}& \rho(y)\leq d^I\rightarrow x\\
&\text{if and only if}& y\leq \phi(d^I\rightarrow x)=d^J\rightarrow \phi(x)\\
&\text{if and only if}& d^J\cdot y\leq \phi(x)\\
&\text{if and only if}& \rho(d^J\cdot y)\leq x.
\end{eqnarray*}
Hence $d^I\cdot\rho(x)=\rho(d^J\cdot x)$ and $g$ is a $\rho$-type mapping.

Conversely, if $g$ is a $\rho$-type mapping then
\begin{eqnarray*}
y\leq d^J\rightarrow \phi(x) &\text{if and only if}& d^J\cdot y\leq \phi(x)\\
&\text{if and only if}& d^I\cdot \rho(y)=\rho(d^J\cdot y)\leq x\\
&\text{if and only if}&  \rho(y)\leq d^I\rightarrow x\\
&\text{if and only if}&  y\leq \phi(d^I\rightarrow x).
\end{eqnarray*}
Hence $d^J\rightarrow \phi(x)=\phi(d^I\rightarrow x)$ and $\phi$ is a $\phi$-type mapping.

Theorem \ref{rep} states that both $\delta$ and $\sigma$ are supremum reversing mappings. Let us assume that $\delta$ is a $\delta$-type mapping. Then
\begin{eqnarray*}
y\leq d^I\rightarrow \epsilon(x) &\text{if and only if}& d^I\cdot y\leq \epsilon(x)\\
&\text{if and only if}& x\leq \delta(d^I\cdot y)=d^J\rightarrow \delta(y)\\
&\text{if and only if}&  d^J\cdot x\leq \delta(y) \\
&\text{if and only if}&  y\leq \epsilon(d^J\cdot x).
\end{eqnarray*}
Hence $d^I\rightarrow \epsilon (x)=\epsilon (d^J\cdot x)$ and $\epsilon$ is $\delta$-type mapping.
\end{proof}

As direct corollary of Theorem \ref{rep} and Lemma \ref{l1} we obtain. 

\begin{cor}\label{cor3}
Let us have a complete residuated lattice $\mathbf A=(A;\vee,\wedge,\cdot,\rightarrow,0,1)$ and let $I$ and $J$ be arbitrary sets. Then
\begin{itemize}
\item[i)] For any $\phi$- type mapping $\phi\colon A^I\longrightarrow A^J$ there exists unique $\rho$-type mapping $\rho\colon A^J\longrightarrow A^I$ such that $\phi$ and $\rho$ form a Galois connection between $A^I$ and $A^J$.
\item[ii)] For any $\rho$- type mapping $\rho\colon A^J\longrightarrow A^I$ there exists unique $\phi$-type mapping $\phi\colon A^I\longrightarrow A^J$ such that $\phi$ and $\rho$ form a Galois connection between $A^I$ and $A^J$.
\item[ii)] For any $\delta$- type mapping $\delta\colon A^I\longrightarrow A^J$ there exists unique $\delta$-type mapping $\epsilon\colon A^J\longrightarrow A^I$ such that $\delta$ and $\epsilon$ form a reversing Galois connection between $A^I$ and $A^J$.
\end{itemize}
\end{cor}

The following theorem explains the definitions stated above. The part that states reversing Galois connections between $\delta_R$ and $\epsilon_R$ is well known \cite{Bel1}. However we include its short proof.

\begin{theorem}\label{op}
Let us have a complete residuated lattice $\mathbf A=(A;\vee,\wedge,\cdot,\rightarrow,0,1)$, let $I$ and $J$ be arbitrary sets and let $R\colon I\times J\longrightarrow A$ be a fuzzy relation. Then 
\begin{itemize}
\item[i)] mappings $\phi_R$ and $\rho_R$ form a Galois connection between $A^I$ and $A^J$. Moreover, the mapping $\phi_R$ is a $\phi$-type mapping and $\rho_R$ is a $\rho$-type mapping.
\item[ii)] mappings $\delta_R$ and $\epsilon_R$ form a reversing Galois connection between $A^I$ and $A^J$. Moreover, both mappings $\delta_R$ and $\epsilon_R$ are $\delta$-type mappings.
\end{itemize}
\end{theorem}
\begin{proof}
For any $x\in A^I$ and $y\in A^J$ we deduce:
\begin{eqnarray*}
y\leq \phi_R(x) &\text{ if and only if }& y(j) \leq \phi_R(x)(j)\quad (\forall j\in J)\\
&\text{ if and only if }& y(j) \leq \bigwedge_{i\in I}(R(i,j)\rightarrow x(i))\quad (\forall j\in J)\\
&\text{ if and only if }& y(j) \leq R(i,j)\rightarrow x(i)\quad (\forall i\in I,\forall j\in J)\\
&\text{ if and only if }& R(i,j)\cdot y(j) \leq x(i)\quad (\forall i\in I,\forall j\in J)\\
&\text{ if and only if }& \bigvee_{j\in J}(R(i,j)\cdot y(j)) \leq x(i)\quad (\forall i\in I)\\
&\text{ if and only if }& \rho_R(y)(i) \leq x(i)\quad (\forall i\in I)\\
&\text{ if and only if }& \rho_R(y) \leq x.
\end{eqnarray*}

Thus $\phi_R$ and $\rho_R$ form a Galois connection between $A^I$ and $A^J$. Consequently  $\phi_R$ is infima preserving and $\rho_R$ is suprema preserving (see Theorem \ref{rez}). Moreover, using Lemma \ref{rez} we have 
$$(d^I\cdot \rho_R(x))(i)=d\cdot\bigvee_{j\in J}(R (i,j)\cdot x(j)) = \bigvee_{j\in J}(R(i,j)\cdot d\cdot x(j))=\rho_R(d^J\cdot x)(i).$$
 for any $d\in A$. Hence $\rho_R$ is a $\rho$-type mapping and, due to Lemma \ref{l1}, also $\phi_R$ is a $\phi$-type mapping.
 
Analogously, if $x\in A^I$ and $y\in A^J$ then we have:
\begin{eqnarray*}
y\leq \delta_R(x) &\text{ if and only if }& y(j) \leq \delta_R(x)(j)\quad (\forall j\in J)\\
&\text{ if and only if }& y(j) \leq \bigwedge_{i\in I}(x(i)\rightarrow R(i,j))\quad (\forall j\in J)\\
&\text{ if and only if }& y(j) \leq x(i)\rightarrow R(i,j)\quad (\forall i\in I,\forall j\in J)\\
&\text{ if and only if }& x(i) \leq y(j)\rightarrow R(i,j)\quad (\forall i\in I,\forall j\in J)\\
&\text{ if and only if }& x(i) \leq \bigwedge_{j\in J}(y(j)\rightarrow R(i,j))\quad (\forall i\in I)\\
&\text{ if and only if }& x(i) \leq \epsilon_R(y)(i)\quad (\forall i\in I)\\
&\text{ if and only if }& x \leq \epsilon_R (y).
\end{eqnarray*}

Thus $\delta_R$ and $\epsilon_R$ form a reversing Galois connection between $A^I$ and $A^J$. Consequently $\delta_R$ and $\epsilon_R$ are a suprema reversing as well (see Theorem \ref{rez}). Moreover, from Lemma \ref{rez} and $x\rightarrow (y \rightarrow z)= (x\cdot y)\rightarrow z$ (see \cite{Haj}) we have 
$$(d^J\rightarrow \delta_R(x))(j)=d\rightarrow\bigwedge_{i\in I}(x(i)\rightarrow R (i,j)) = \bigwedge_{i\in I}((d\cdot x(i))\rightarrow R (i,j))=\delta_R(d^I\cdot x)(j).$$
 for any $d\in A$. Hence $\delta_R$ is a $\delta$-type mapping and, due to Lemma \ref{l1}, also $\epsilon_R$ is $\delta$-type mapping.
\end{proof}

\section{Representation theorems}

In this section we will present the main results of the paper. In the following theorem we characterize $\phi$, $\rho$ and $\delta$-type mappings as mappings induced by fuzzy relations. A result similar to the second part has been proved for quantals by S. Solovjovs \cite{Sol}.

\begin{theorem}\label{hlavni}
Let us have a complete residuated lattice $\mathbf A=(A;\vee,\wedge,\cdot,\rightarrow,0,1)$ and let $I$ and $J$ be arbitrary sets. 
\begin{itemize}
\item[i)] Let us have a mapping $\phi\colon A^I\longrightarrow A^J$ (resp. $\rho\colon A^J\longrightarrow A^I$). Then $\phi$ is a $\phi$-type mapping resp. $\rho$ is a $\rho$-type mapping) if and only if there exists an $\mathbf A$-relation $R\colon$ $I\times J\longrightarrow A$ such that $\phi=\phi_R$ (or $\rho=\rho_R$). Moreover, a $\phi$-type mapping and a $\rho$-type mapping form a Galois connection if and only if they are induced by same fuzzy relation. 

\item[ii)] Let us have a mapping $\delta\colon A^I\longrightarrow A^J$ (or $\epsilon\colon A^J\longrightarrow A^I$). Then $\delta$ is a $\delta$-type mapping (or $\epsilon$ is a $\delta$-type mapping) if and only if there exists an $\mathbf A$-relation $R\colon I\times J\longrightarrow A$ such that $\delta=\delta_R$ (or $\epsilon=\epsilon_R$). Moreover, a $\delta$-type mapping and a converse $\delta$-type mapping form a reversing Galois connection if and only if they are induced by same fuzzy relation.
\end{itemize}
\end{theorem}
\begin{proof}
{\it i)} Let us have a $\phi$-type mapping $\phi\colon  A^I\longrightarrow A^J$. And let us define
$$R (i,j)=\bigwedge_{a\in A^I}(\phi(a)(j)\rightarrow a(i)).\eqno{(R_\phi)}$$  Clearly, the inequality $R(i,j) \leq \phi(x)(j)\rightarrow x(i)$ holds for any $x\in A^I$, $i\in I$, and $j\in J$. Thus 
\begin{eqnarray*}
\phi_R(x)(j)&=&\bigwedge_{i\in I} (R(i,j)\rightarrow x(i))\\
&\geq& \bigwedge_{i\in X}((\phi(x)(j)\rightarrow x(i))\rightarrow x(i))\\
&\geq& \phi(x)(j).
\end{eqnarray*}
Let us define a set
$$M(i,j)=\{x(i)\mid \phi(x)(j)=1\}.$$
We can state the following claim.  
\begin{cl}\label{cl1}
$R (i,j)=\bigwedge M(i,j)$.
\end{cl}
\begin{proof}
If $x(i)\in M(i,j)$ then $\phi(x)(j)=1$ and, consequently, $R (i,j)=\bigwedge_{a\in A^I}(\phi(a)(j)\rightarrow a(i))\leq \phi(x)(j)\rightarrow x(i)=x(i)$. Hence $R (i,j)\leq \bigwedge M(i,j)$.

On the other hand, let $a\in A^I$, $i\in I$ and $j\in J$. Denoting $f=\phi(a)(j)$ we obtain $1=f\rightarrow \phi(a)(j)=(f^J\rightarrow \phi(a))(j)=\phi(f^I\rightarrow a)(j)$. Thus $\phi(a)(j)\rightarrow a(i)=f\rightarrow a(i)=(f^I\rightarrow a)(i)\in M(i,j)$ which implies $\bigwedge M(i,j)\leq \bigwedge_{a\in A^I} (\phi(a)(j)\rightarrow a(i))=R(i,j)$.
\end{proof}

\begin{cl}\label{cl2}
The set $M_j=\{x\in A^I\mid \phi(x)(j)=1\}$ is a lattice filter on lattice reduct of $\mathbf A^I$ closed on (infinite) infima. Consequently, there exists a least element $m\in M_j$.
\end{cl}
\begin{proof}
Since $1\in M_j$, we have $\emptyset\not=M_j$. Since $M\subseteq M_j$. Because $A^J$ is a complete lattice and $\phi$ is infimum preserving, we have
$$\phi\big(\bigwedge M\big)(j)=\big(\bigwedge \{\phi(x)\mid x\in M\}\big)(j)=\bigwedge\{\phi(x)(j)\mid x\in M\}=1$$ and $\bigwedge M\in M_j$.
\end{proof}
Let $j\in J$. Claim \ref{cl2} prove an existence of the least element $m\in M_j$ and, moreover,
$$m(i)=\big(\bigwedge M_j \big )(i)=\bigwedge \{x(i)\mid \phi(x)(j)\}=\bigwedge M(i,j).$$
Let $x\in A^I$. From Claim \ref{cl1} we obtain
$\phi_R(x)(j)=\bigwedge_{i\in I} (m(i)\rightarrow x(i))$ and denoting $f=\phi_R(x)(j)$ we obtain $f\leq m(i)\rightarrow x(i)$ for any $i\in I$. Due to the adjointness property we have $m(i)\leq f\rightarrow x(i)=(f^I\rightarrow x)(i)$ for any $i\in I$. Then $m\leq f^I\rightarrow x$ and $f^I\rightarrow x\in M_j$. Using the definition of the set $M_j$ we have proved that $$1=\phi(f^I\rightarrow x)(j)=(f^J\rightarrow \phi(x))(j)=f\rightarrow \phi(x)(j)=\phi_R(x)(j)\rightarrow \phi(x)(j)$$ and consequently $\phi_R(x)(j)\leq \phi(x)(j)$ for any $x\in A^I$ and any $j\in J$.

Altogether, $\phi=\phi_R$ holds. If $\phi$ and $\rho$ form a Galois connection then, due to Theorems \ref{rep} and \ref{op}, we obtain $\rho=\rho_R$.

{\it ii)} Let us have a $\delta$-type mapping $\delta\colon A^I\longrightarrow A^J$. We denote $$R(i,j)=\bigvee_{a\in A^I}(\delta(a)(j)\cdot a(i)).\eqno{(R_\delta)}$$
Since $R(i,j)\geq \delta(x)(j)\cdot x(i)$ holds for all $x\in A^i$, we obtain 
\begin{eqnarray*}
\delta_R(x)(j)&=&\bigwedge_{i\in I}(x(i)\rightarrow R(i,j))\\
&\geq & \bigwedge_{i\in I}(x(i)\rightarrow (\delta(x)(j)\cdot x(i))\\
&\geq &\delta (x)(j).
\end{eqnarray*} 
Analogously to previous part we denote
$$M(i,j)=\{x(i)\mid \delta (x)(j)=1\}.$$
\begin{cl}\label{c3}
$R(i,j)=\bigvee M(i,j)$.
\end{cl}
\begin{proof}
If $x(i)\in M(i,j)$ then $\delta(x)(j)=1$ and, consequently, $R (i,j)=\bigvee_{a\in A^I}(a(i)\cdot \delta(a)(j))\geq x(i)\cdot \delta(a)(j)=x(i)$. Hence $R (i,j)\geq \bigvee M(i,j)$.

On the other hand, let $a\in A^I$, $i\in I$ and $j\in J$. Denoting $f=\delta(a)(j)$ we obtain $1=h\rightarrow \delta(a)(j)=(h^J\rightarrow \delta(a))(j)=\delta(h^I\cdot a)(j)$. Thus $\delta(a)(j)\cdot a(i)=h\cdot a(i)=(h^I\cdot a)(i)\in M(i,j)$ and $\bigvee M(i,j)\geq \bigvee_{a\in A^I} (\delta(a)(j)\cdot a(i))=R(i,j)$.
\end{proof}

\begin{cl}\label{cl4}
The set $M_j=\{x\in A^I\mid \delta(x)(j)=1\}$ is a lattice ideal on a lattice reduct of $\mathbf A^I$ closed on (infinite) suprema. Consequently, there exists a greatest element $m\in M_j$.
\end{cl}
\begin{proof}
Clearly $0\in M_j$ and therefore $M_j\not=\emptyset$. Let $M\subseteq M_j$. Since $A^J$ is a complete lattice and $\delta$ is suprema reversing, we have
$$\delta\big(\bigvee M\big)(j)=\big(\bigwedge \{\delta(x)\mid x\in M\}\big)(j)=\bigwedge\{\delta(x)(j)\mid x\in M\}=1$$ and $\bigvee M\in M_j$.
\end{proof}

Let $j\in J$. Claim \ref{cl4} prove an existence of the greatest element $m\in M_j$ and, moreover,
$$m(i)=\big(\bigvee M_j \big )(i)=\bigvee \{x(i)\mid \delta(x)(j)\}=\bigvee M(i,j).$$
Let $x\in A^I$ from Claim \ref{c3} we obtain
$\delta_R(x)(j)=\bigwedge_{i\in I} (x(i)\rightarrow m(i))$ and thus denoting $h=\delta_R(x)(j)$ we obtain $h\leq x(i)\rightarrow m(i)$ for any $i\in I$. Due to the adjointness property we have also $m(i)\geq h\cdot x(i)=(h^I\cdot x)(i)$ for any $i\in I$. Then $m\geq h^I\cdot x$ and $h^I\cdot x\in M_j$. Using the definition of the set $M_j$ we have proved that $$1=\delta(h^I\cdot x)(j)=(h^J\rightarrow \delta(x))(j)=h\rightarrow \delta(x)(j)=\delta_R(x)(j)\rightarrow \delta(x)(j)$$ and consequently $\delta_R(x)(j)\leq \delta(x)(j)$ for any $x\in A^I$ and $j\in J$.

Altogether, $\delta=\delta_R$ holds. If $\delta$ and $\epsilon$ form a reversing Galois connection then, due to Theorems \ref{rep} and \ref{op}, we obtain $\epsilon=\epsilon_R$.
\end{proof}

\begin{lemma}\label{interclosezpet}
Let us have a $\phi$-type mapping $\phi\colon A^I\longrightarrow A^J$ and a $\rho$-type mapping $\rho\colon A^J\longrightarrow A^I$ forming a Galois connection where $\mathbf A=(A;\vee, \wedge,\cdot,\rightarrow,0,1)$ is a residuated lattice. Then $\rho\phi$ is an interior operator on $\mathbf A^I$ satisfying $d^I\cdot \rho\phi (x)\leq \rho\phi (d^I\cdot x)$. Moreover, $\phi\rho$ is a closure operator on $\mathbf A^J$ satisfying $d^J\cdot \phi\rho (x)\leq \phi\rho (d^J\cdot x)$.

Let us have $\delta$-type mappings $\delta\colon A^I\longrightarrow A^J$ and  $\epsilon\colon A^J\longrightarrow A^I$ forming a reversing Galois connection. Then $\epsilon\delta$ is a closure operator on $\mathbf A^I$ satisfying $d^I\cdot\epsilon \delta\phi (x)\leq \epsilon\delta (d^I\cdot x)$.
\end{lemma}
\begin{proof}
Corollary \ref{c1} shows that $\rho\phi$ is an interior operator and that $\phi\rho$ is a closure operator. Let $d\in A$ and $x\in A^I$. Then $\phi (x)\leq \phi(d^I\rightarrow (d^I\cdot x))=d^J\rightarrow \phi(d^I\cdot x)$. Using the adjointness property we obtain $d^J\cdot \phi(x)\leq \phi (d^I\cdot x)$ and, consequently, 
$$d^I\cdot \rho \phi(x) =  \rho(d^J\cdot \phi(x))\leq \rho\phi (d^I\cdot x)$$ and
$$d^J\cdot \phi \rho(x) \leq   \phi(d^J\cdot \rho(x)) = \phi\rho (d^I\cdot x).$$

Analogously, Corollary \ref{c2} states that $\epsilon\delta$ is a closure operator. Moreover, using the antitonicity of $\epsilon$ we obtain $d^I\rightarrow \epsilon (d^J\rightarrow x)=\epsilon (d^J\cdot(d^J\rightarrow x))\geq \epsilon (x)$. Hence $d^I\cdot \epsilon (x)\leq \epsilon (d^J\rightarrow x)$ and
$$d^I\cdot \epsilon\delta (x)\leq \epsilon (d^J\rightarrow\delta (x) )=\epsilon\delta (d^I\cdot x).$$
\end{proof}

\begin{theorem}\label{inter}
If $\mathbf A=(A;\vee,\wedge,\cdot ,\rightarrow,0,1)$ is a residuated lattice and $I$ is an arbitrary set then any interior operator $O\colon A^I\longrightarrow A^I$ satisfying $d^I\cdot O(x)\leq O(d^I\cdot x)$ can be expressed by $O=\phi_R\rho_R$ where $R$ is a fuzzy relation.
\end{theorem}
\begin{proof}
Let $J=\{O(a)\mid a\in A^I\}$ be a set and $R\colon I\times J\longrightarrow A$ a fuzzy relation given by $R(i,O(a))=O(a)(i)$. Then
\begin{eqnarray*}
\rho_R(x)(i)&=& \bigvee_{O(a)\in J}(R(i,O(a))\cdot x(O(a)))\\
&=& \bigvee_{O(a)\in J}(O(a)(i)\cdot x(O(a)))\\
&=& \bigvee_{O(a)\in J}(O(a)\cdot x(O(a))^I)(i)\\
&=& \big(\bigvee_{O(a)\in J}(O(a)\cdot x(O(a))^I)\big)(i)\\
\end{eqnarray*}
and consequently we obtain $$\rho_R(x)=\bigvee_{O(a)\in J}(O(a)\cdot x(O(a))^I).$$
Then 
\begin{eqnarray*}
\rho_R(x)&\geq& O(\rho_R(x))\\
&=&O\big(\bigvee_{O(a)\in J}(O(a)\cdot x(O(a))^I)\big)\\
&\geq & \bigvee_{O(a)\in J}O((O(a)\cdot x(O(a))^I))\\
&\geq & \bigvee_{O(a)\in J}(OO(a)\cdot x(O(a))^I)\\
&= & \bigvee_{O(a)\in J}(O(a)\cdot x(O(a))^I)\\
&=&\rho_R(x)
\end{eqnarray*}
proves $O(\rho_R(x))=\rho_R(x)$ for all $x\in A^J$.

On the other side, let us have any closed set $O(x)\in J$ then defining $y\in A^J$ by
\begin{center}
$y(O(a))=\left\{
\begin{array}{ll}
1 &\mbox{ if } O(a)=O(x)\\
0 &\mbox{ if } O(a)\not =O(x),
\end{array}\right. $
\end{center}
we have
$$\rho_R(y)= \bigvee_{O(a)\in J}(O(a)\cdot y(O(a))^I) =O(x).$$ Altogether, we have proved that interiors of the operator $O$ are just elements $\rho_R(y)$ such that $y\in A^J$. With respect to Corollary \ref{c1} the operators $O$ and $\rho_R\phi_R$ have the same systems of interiors and thus $O =\rho_R\phi_R$.
\end{proof}

\begin{theorem}\label{clos}
If $\mathbf A=(A;\vee,\wedge,\cdot ,\rightarrow,0,1)$ is a residuated lattice and $I$ is an arbitrary set then any closure operator $O\colon A^I\longrightarrow A^I$ satisfying $d^I\cdot O(x)\leq O(d^I\cdot x)$ can be expressed by $O=\epsilon_R\delta_R$ where $R$ is a fuzzy relation.
\end{theorem}
\begin{proof}
Let $J=\{O(a)\mid a\in A^I\}$ be a set and $R\colon I\times J\longrightarrow A$ be a a fuzzy relation given by $R(i,O(a))=O(a)(i)$. Then
\begin{eqnarray*}
\epsilon_R(x)(i)&=& \bigwedge_{O(a)\in J}(x(O(a))\rightarrow R(i,O(a)))\\
&=& \bigwedge_{O(a)\in J}(x(O(a)) \rightarrow O(a)(i))\\
&=& \bigwedge_{O(a)\in J}(x(O(a))^I\rightarrow O(a))(i)\\
&=& \big(\bigwedge_{O(a)\in J}(x(O(a))^I\rightarrow O(a))\big)(i)\\
\end{eqnarray*}
and, consequently, $$\epsilon_R(x)=\bigwedge_{O(a)\in J}(x(O(a))^I\rightarrow O(a)).$$
Then 
\begin{eqnarray*}
\epsilon_R(x)&\leq& O(\epsilon_R(x))\\
&=&O\big(\bigwedge_{O(a)\in J}(x(O(a))^I\rightarrow O(a))\big)\\
&\leq & \bigwedge_{O(a)\in J}O(x(O(a))^I\rightarrow O(a))\\
&\leq & \bigwedge_{O(a)\in J}(x(O(a))^I \rightarrow OO(a))\\
&=& \bigwedge_{O(a)\in J}(x(O(a))^I \rightarrow O(a))\\
&=&\epsilon_R(x)
\end{eqnarray*}
proves $O\rho_R(x)=\rho_R(x)$ for all $x\in A^J$.

On the other side, let us have any closed set $O(x)\in J$ then defining $y\in A^J$ by
\begin{center}
$y(O(a))=\left\{
\begin{array}{ll}
1 &\mbox{ if } O(a)=O(x)\\
0 &\mbox{ if } O(a)\not =O(x).
\end{array}\right. $
\end{center}
and then
$$\epsilon_R(y)= \bigwedge_{O(a)\in J}(y(O(a))^I\rightarrow O(a)) =O(x)$$ holds. Altogether, we have proved that interiors of the operator $O$ are just all $\epsilon(y)$ such that $y\in A^J$. With respect to Corollary \ref{c1} the operators $O$ and $\epsilon_R\delta_R$ have same systems of interiors and thus $O =\epsilon_R\delta_R$.
\end{proof}

\section{Application of the representation theorems in Pavelka's logic}

The Pavelka's logic enriches the \L ukasiewicz's multi-valued logic about an infinite systems of constants belonging into a set $[0,1]\cap \mathbb Q$ (see \cite{Pav}). These  constants form a dense subalgebra of the \emph{standard MV-algebra} defined as $([0,1];\oplus,\neg,0)$ where
$$x\oplus y=\min\{x+y,1\},$$ $$\neg x= 1-x$$
for all $x,y\in[0,1]$. We are going to define the algebraic model of Pavelka's logic.
\begin{defin}  
By a \emph{Pavelka's algebra} we mean an algebra $\mathbf A=(A;\oplus,\neg,\{\mathbf r\mid r\in [0,1]\cap \mathbb Q\})$ satisfying
\begin{itemize}
\item[(P1)] the reduct $(A;\oplus,\neg, \mathbf 0)$ is an MV-algebra,
\item[(P2)] if $r,s,t\in [0,1]\cap\mathbb Q$ are such that $r\oplus s = t$ (computed in $[0,1]\cap \mathbb Q$) then $\mathbf r \oplus \mathbf s = \mathbf t$ (computed in $\mathbf A$),
\item[(P3)] if $r,s\in [0,1]\cap \mathbb Q$ are such that $\neg r =s$ (computed in $[0,1]\cap \mathbb Q$) then $\neg\mathbf r =\mathbf s$ (computed in $\mathbf A$).
\end{itemize}
We note that an element of Pavelka's algebra will be typeset in bold if and only if it is a constant.
\end{defin}

Let $\mathbf A=(A;\oplus,\neg,\{\mathbf r\mid r\in [0,1]\cap \mathbb Q\})$ be a Pavelka's algebra. By a \emph{filter} of $\mathbf A$ we mean a filter of the MV-reduct $(A;\oplus,\neg,\mathbf 0)$. If $F$ is a proper filter then the only constant in $F$ is $\mathbf 1$ (see \cite{Geo}). Thus there is a one-to-one correspondence between filters on a Pavelka's algebra and its congruences. An MV-algebra is \emph{semisimple} if it is a subdirect product of simple MV-algebras (having only trivial proper filter). It is well known that semisimple MV-algebras are isomorphic to subalgebras of powers of the standard MV-algebra.

We say that a filter is a \emph{maximal} if it is maximal proper filter. Since any simple MV-algebra, and thus also any simple Pavelka's algebra,  is uniquely embeddable into the standard MV-algebra (= standard Pavelka's algebra)  (see \cite{Mun}), we will identify  in this embedding the elements of (any) simple MV-algebra with theirs images. Especially, if $F$ is a maximal filter on a Pavelka's algebra $\mathbf A$ then $\mathbf A/F$ is a simple Pavelka's MV-algebra and thus an algebra $\mathbf A/F$ will be considered as a subalgebra of the standard MV-algebra ($\mathbf A/F\subseteq [0,1]$). 

With respect to the previous paragraph, we denote the a set of all maximal filters of a Pavelka's algebra $\mathbf A$ by $\s (\mathbf A)$. Observe that any Pavelka's algebra can be embedded into $[0,1]^{\s\mathbf A}$ by
$$x(F)=x/F$$ for all $x\in A$ and $F\in\s \mathbf A$.  We call this embedding $n_\mathbf A\colon \mathbf A\longrightarrow [0,1]^{\s \mathbf A}$ the \emph{natural embedding} .

\begin{defin}
Let $\mathbf A=(A;\oplus,\neg,\{\mathbf r\mid r\in [0,1]\cap \mathbb Q\})$ and $\mathbf B=(B;\oplus,\neg,\{\mathbf r\mid r\in [0,1]\cap \mathbb Q\})$ be Pavelka's algebras. Let $f\colon A\longrightarrow B$ and $g\colon B\longrightarrow A$ be mappings forming a Galois connection between $A$ and $B$.We say that $f$ is a \emph{strong left adjoint of $g$}, and we denote it by $f\leftrightarrows g$, if one of the following equivalent conditions holds 
\begin{itemize}
\item[(f)] $\mathbf r\rightarrow f(x)=f(\mathbf r\rightarrow x)$ for all $x,\mathbf r\in A$,
\item[(g)] $\mathbf r\cdot g(x)=g(\mathbf r\cdot x)$ for all $x,\mathbf r\in A$
\end{itemize}

Let $d\colon A\longrightarrow g$ and $h\colon B\longrightarrow A$ be mappings forming a reversing Galois connection between $A$ and $B$. We say that $d$ is \emph{strong reversing left adjoint of $g$}, and we denote it by $f\rightleftharpoons g$, if one of the following equivalent conditions holds.
\begin{itemize}
\item[(d)] $\mathbf r\rightarrow d(x)=d(\mathbf r\cdot x)$ for all $x,\mathbf r\in A$,
\item[(h)] $\mathbf r\rightarrow h(x)=h(\mathbf r\cdot x)$ for all $x,\mathbf r\in A$
\end{itemize}
\end{defin}

The equivalence of the conditions (f) and (g) (or (d) and (h)) is an easy analogy of Corollary~\ref{cor3}. The following theorems are direct analogies of Theorems \ref{op} and \ref{hlavni}. 

\begin{theorem}\label{pt1}
i) If $R\colon I\times J\longrightarrow [0,1]$ is a fuzzy relation then mappings $f_R\colon [0,1]^I\longrightarrow [0,1]^J$ and $g_R\colon [0,1]^J\longrightarrow [0,1]^I$ defined by
$$f_R(x)(j)=\bigwedge_{i\in I} (R(i,j)\rightarrow x(i))\mbox{ for all } x\in A, j\in J, \eqno{(f_R)}$$
$$g_R(x)(i)=\bigvee_{j\in J} (R(i,j)\cdot x(j))\mbox{ for all } x\in B, i\in I, \eqno{(g_R)}$$
satisfy $f_R\leftrightarrows g_R$.

ii) Let us have semisimple Pavelka's algebras $\mathbf A=(A;\oplus,\neg,\{\mathbf r\mid r\in [0,1]\cap \mathbb Q\})$ and $\mathbf B=(B;\oplus,\neg,\{\mathbf r\mid r\in [0,1]\cap \mathbb Q\})$ and mappings $f\colon A\longrightarrow B$ and $g\colon B\longrightarrow A$ satisfying $f\leftrightarrows g$. Then there exists a fuzzy relation
$$R\colon \s \mathbf A\times \s\mathbf B\longrightarrow [0,1]$$ 
such that the diagram
$$
\bfig
\square/>`{^{(}->}`{^{(}->}`<-/<1000,500>[A`B`\lbrack0,1\rbrack^{\s\mathbf A}`\lbrack0,1 \rbrack^{\s\mathbf B};f`n_{\mathbf A}`n_{\mathbf B}` g_R]
\morphism(0,500)|b|/@{<-}@<-5pt>/<1000,0>[A`B;g]
\morphism(0,0)|a|/@{>}@<5pt>/<1000,0>[\lbrack 0,1\rbrack^{\s \mathbf A}`\lbrack 0,1\rbrack^{\s \mathbf B};f_R]
\Loop(0,500)A(ur,ul)_{gf}
\Loop(1000,500)B(ur,ul)_{fg}
\Loop(0,0)\lbrack 0,1\rbrack^{\s \mathbf A}(dl,dr)_{g_Rf_R}
\Loop(1000,0)\lbrack 0,1\rbrack^{\s \mathbf B}(dl,dr)_{f_Rg_R}
\efig$$
commutes.
\end{theorem}
\begin{proof}
The part i) was proved in Theorem \ref{op} since the constants are, clearly, diagonal elements. Converse part of the proof is similar. For given $f\leftrightarrows g$ we define binary relation $$R(F,G)=\bigwedge_{a\in A}(f(a)/G\rightarrow a/F).$$  Then $R(F,G)\leq f(x)/G\rightarrow x/F$ for all $x\in A$ yields $$f_R(x)/G =\bigwedge_{F\in\s\mathbf A}(R(F,G)\rightarrow x/F)\geq f(x)/G$$ for all $x\in A$ and $G\in \s\mathbf B$. Hence $f_R(x)\geq f (x)$.

Let us define the set
$$M(F,G)=\{x/F\mid f(x)/G=1\}.$$
We can state  the following claim. 

\begin{cl}\label{cl5}
$R (F,G)=\bigwedge M(F,G)$.
\end{cl}
\begin{proof}
If $x/F\in M(F,G)$ then $f(x)/G=1$ and consequently $R (i,j)\leq f(x)/G\rightarrow x/F=x/F$. Consequently, $R (F,G)\leq \bigwedge M(F,G)$ holds.
 
If $\mathbf r\leq f(x)/G$ then $1=\mathbf r\rightarrow f(x)/G=f(\mathbf r\rightarrow x)/G$ and thus $\mathbf r\rightarrow x/F=(\mathbf r\rightarrow x)/F\in M(F,G)$. Consequently we obtain $$\bigwedge M(F,G)\leq \bigwedge_{\mathbf r\leq f(x)/G}(\mathbf r\rightarrow x/F)=\big(\bigvee_{\mathbf r\leq f(x)/G}\mathbf r\big)\rightarrow x/F=f(x)/G\rightarrow x/F$$
for all $x\in A$. Hence $\bigwedge M(F,G)\leq \bigwedge_{x\in A}(f(x)/G\rightarrow x/G)=R(F,G)$. 
\end{proof}

Analogously to Claim \ref{c2}, we can prove that the set $M_G=\{x\in A\mid f(x)/G=1\}$ has the least element $m_G$ for all $G\in \s\mathbf B$. Consequently, $m_G/F$ is the least element of $M(F,G)$ and thus $m_G/F=R(F,G)$.

Let $x\in A$; using Claim \ref{cl5} we obtain
$f_R(x)/G=\bigwedge_{F\in \s \mathbf A} (m_G/F\rightarrow x/F)$ and thus for all $\mathbf r\leq f_R(x)/G$ we obtain $m_G/F\leq \mathbf r\rightarrow x/F=(\mathbf r\rightarrow x)/F$ for any $F\in\s \mathbf A$. Therefore, thus $m\leq \mathbf r\rightarrow x$ and $\mathbf r\rightarrow x\in M_G$. Using the definition of the set $M_G$ we have proved that $$1=f(\mathbf r\rightarrow x)/F=\mathbf r \rightarrow f(x)/F$$ and consequently $\mathbf r\leq f(x)/G$.

The inequality $\mathbf r\leq f_R(x)/G$ implies $\mathbf r\leq f(x)/G$ which gives $f_R(x)/G\leq f(x)/G$ for all $x\in A$ and $G\in\s \mathbf B$. Therefore, $f=f_R$. If $f$ and $g$ form a Galois connection then due to Theorems \ref{rep} and \ref{op} we obtain $g=g_R$.
\end{proof}

\begin{theorem}\label{pt2}
i) If $R\colon I\times J\longrightarrow [0,1]$ is a fuzzy relation then mappings $d_R\colon [0,1]^I\longrightarrow [0,1]^J$ and $h_R\colon [0,1]^J\longrightarrow [0,1]^I$ defined by
$$d_R(x)(j)=\bigwedge_{i\in I} (x(i)\rightarrow R(i,j))\mbox{ for all } x\in A, j\in J, \eqno{(d_R)}$$
$$h_R(x)(i)=\bigwedge_{j\in J} (x(j)\rightarrow R(i,j))\mbox{ for all } x\in B, i\in I, \eqno{(h_R)}$$
satisfy $d_R\rightleftharpoons h_R$.

ii) Let us have semisimple Pavelka's algebras $\mathbf A=(A;\oplus,\neg,\{\mathbf r\mid r\in [0,1]\cap \mathbb Q\})$ and $\mathbf B=(B;\oplus,\neg,\{\mathbf r\mid r\in [0,1]\cap \mathbb Q\})$ and mappings $d\colon A\longrightarrow B$ and $h\colon B\longrightarrow A$ satisfying $d\rightleftharpoons h$. Then there exists a fuzzy relation
$$R\colon \s \mathbf A\times \s\mathbf B\longrightarrow [0,1]$$ 
such that the diagram
$$
\bfig
\square/>`{^{(}->}`{^{(}->}`<-/<1000,500>[A`B`\lbrack 0,1\rbrack^{\s\mathbf A}`\lbrack 0,1\rbrack^{\s\mathbf B};d`n_{\mathbf A}`n_{\mathbf B}` h_R]
\morphism(0,500)|b|/@{<-}@<-5pt>/<1000,0>[A`B;h]
\morphism(0,0)|a|/@{>}@<5pt>/<1000,0>[\lbrack0,1\rbrack ^{\s \mathbf A}`\lbrack 0,1\rbrack^{\s \mathbf B};d_R]
\Loop(0,500)A(ur,ul)_{hd}
\Loop(1000,500)B(ur,ul)_{dh}
\Loop(0,0)\lbrack 0,1\rbrack^{\s \mathbf A}(dl,dr)_{h_Rd_R}
\Loop(1000,0)\lbrack 0,1\rbrack^{\s \mathbf B}(dl,dr)_{d_Rh_R}
\efig$$
commutes.
\end{theorem}

\begin{proof}
The part i) was proved in the Theorem \ref{op} since constants are diagonal elements. The converse part of the proof is similar. For given $d\rightleftharpoons h$ we define a binary relation $$R(F,G)=\bigvee_{a\in A}(d(a)/G\cdot a/F).$$  Then the inequality $R(F,G)\geq d(a)/G\cdot a/F$ holds for all $a\in A$ and this yields $d_R(x)/G \geq d(x)/G$ for all $G\in \s\mathbf B$. Hence $d_R(x)\geq d (x)$.

Let us define the set
$$M(F,G)=\{x/F\mid d(x)/G=1\}.$$
Then the following claim we can state 

\begin{cl}\label{cl6}
$R (F,G)=\bigvee M(F,G)$.
\end{cl}
\begin{proof}
If $x/F\in M(F,G)$ then $d(x)/G=1$ and consequently $R (i,j)\geq d(x)/G\cdot x/F=x/F$. Consequently, $R (F,G)\geq \bigwedge M(F,G)$.
 
If $\mathbf r\leq d(x)/G$ then $1=\mathbf r\rightarrow f(x)/G=d(\mathbf r\cdot x)/G$ and thus $\mathbf r\cdot x/F=(\mathbf r\cdot x)/F\in M(F,G)$. Consequently we obtain $$\bigvee M(F,G)\geq \bigvee_{\mathbf r\leq f(x)/G}(\mathbf r\cdot x/F)=\big(\bigvee_{\mathbf r\leq f(x)/G}\mathbf r\big)\cdot x/F=f(x)/G\cdot x/F$$
for all $x\in A$. Hence $\bigvee M(F,G)\leq \bigvee_{x\in A}(f(x)/G\rightarrow x/G)=R(F,G)$ holds. 
\end{proof}

Analogously to Claim \ref{c2} we can prove that the set $M_G=\{x\in A\mid d(x)/G=1\}$ has the greatest element $m_G$ for all $G\in \s\mathbf B$. Consequently, $m_G/F$ is the greatest element of $M(F,G)$ and thus $m_G/F=R(F,G)$.

Let $x\in A$; using Claim \ref{cl5}, we obtain
$d_R(x)/G=\bigwedge_{F\in \s \mathbf A} (x/F\rightarrow m_G/F)$ and thus for all $\mathbf r\leq d_R(x)/G$ we have $(\mathbf r\cdot x)/F=\mathbf r\cdot x/F\leq m_G/F$ for any $F\in\s \mathbf A$. Therefore $\mathbf r\cdot x\leq m$ and $\mathbf r\cdot x\in M_G$ hold. Using the definition of the set $M_G$ we have proved that $$1=d(\mathbf r\cdot x)/F=\mathbf r \rightarrow d(x)/F$$ and consequently $\mathbf r\leq d(x)/G$.

The inequality $\mathbf r\leq d_R(x)/G$ implies $\mathbf r\leq d(x)/G$ which gives $d_R(x)/G\leq d(x)/G$ for all $x\in A$ and $G\in\s \mathbf B$. Therefore, $d=d_R$. If $d$ and $h$ form a reversing Galois connection then due to Theorems \ref{rep} and \ref{op} we obtain $h=h_R$.
\end{proof}

The followings theorems state simple but important facts satisfied by Pavelka's algebras (or semisimple Pavelka's algebras).

\begin{theorem}\label{otocenesipky}
Let $\mathbf A=(A;\oplus,\neg,\{\mathbf r\mid r\in [0,1]\cap \mathbb Q\})$ and $\mathbf B=(B;\oplus,\neg,\{\mathbf r\mid r\in [0,1]\cap \mathbb Q\})$ be a Pavelka's algebras and let us have mappings $f\colon A\longrightarrow B$ and $g\colon B\longrightarrow A$. Then  $f\leftrightarrows g$ if and only if $(\neg g\neg)\leftrightarrows (\neg f \neg)$. 
\end{theorem}
\begin{proof}
Assume that $f\leftrightarrows g$. Using the double negation law and the adjointness property we obtain that $x\leq \neg g(\neg y)$ if, and only if, $\neg f(\neg x) \leq y$.
The rest of the proof follows from $x\rightarrow y = \neg (x\cdot \neg y)$ and $x\cdot y = \neg(x\rightarrow \neg y)$ (see \cite{Mun}). Particularly, $\mathbf r\cdot \neg f (\neg x)= \neg (\mathbf r\rightarrow f (\neg x))=\neg (f(\mathbf r \rightarrow \neg x))= \neg f \neg (\mathbf r \cdot x)$.
\end{proof}

\begin{theorem}\label{konjug}
Let $R\colon I\times J\longrightarrow [0,1]$ be a fuzzy relation. Then $\neg f_R\neg = g_{R^{-1}}$ and $\neg g_R\neg = f_{R^{-1}}$.
\end{theorem}
\begin{proof}
We have
\begin{eqnarray*}
\neg f_R(\neg x)(j)&=& \neg\bigwedge_{i\in I} (R(i,j)\rightarrow \neg x(i))\\
&=&  \bigvee_{i\in I} (R^{-1}(j,i)\cdot x(i))\\
&=& g_{R^{-1}}(x)(j).
\end{eqnarray*}
The proof of the second statement is analogous. 
\end{proof}

\begin{theorem}\label{boolean}
i) If a fuzzy relation $R\colon I\times J\longrightarrow [0,1]$ is boolean ($\{0,1\}$-valued) then $f_R(x)\cdot f_R(y)\leq f_R(x\cdot y)$ for all $x,y\in A^I$.

ii) Let $\mathbf A=(A;\oplus,\neg,\{\mathbf r\mid r\in [0,1]\cap \mathbb Q\})$ and $\mathbf B=(B;\oplus,\neg,\{\mathbf r\mid r\in [0,1]\cap \mathbb Q\})$ be a semisimple Pavelka's algebras and let  $f\colon A\longrightarrow B$ and $g\colon B\longrightarrow A$ be a mappings satisfying $f\leftrightarrows g$. If $f_R(x)\cdot f_R(y)\leq f_R(x\cdot y)$ then the relation $R$ from Theorem \ref{pt1} is boolean.
\end{theorem}
\begin{proof}
i) We remark that  MV-product preserves existing infima and suprema (see \cite{Mun}). Thus if $R$ is a boolean relation then
$$f_R(x)(j)=\bigwedge_{i\in I}(R(i,j)\rightarrow x (i))= \bigwedge_{iRj} x (i).$$ 
Thus we can compute
\begin{eqnarray*}
f_R(x)(j)\cdot f_R(y)(j) &=& \big( \bigwedge_{iRj} x (i)\big)\cdot \big( \bigwedge_{iRj} y (i)\big)\\
&=& \bigwedge_{i_1,i_2 Rj} (x(i_1)\cdot y(i_2))\\
&\leq & \bigwedge_{iRj} (x\cdot y)(i)\\
&=& f_R(x\cdot y)(j).
\end{eqnarray*} 

If $f_R(x)\cdot f_R(y)\leq f_R(x\cdot y)$ then the set $M_G$ from Theorem \ref{pt1} is a filter. Thus, if $x,y\in M_G$ then $1=f_R(x)(j)\cdot f_R(y)(j)\leq f_R(x\cdot y)$ yields $x\cdot y\in M_G$.  Hence the least element $m_G\in M_G$ is an idempotent and $m_G/F\in [0,1]$ must be an idempotent too. However the standard MV-algebra possesses only two idempotents 0 and 1. Therefore $R(F,G)=m_G/F$ is a boolean relation.
\end{proof}
 
\subsection{Tense operators on Pavelka's algebra} We recall that a structure $(\mathbf B, G,H)$ is a tense Boolean algebra if $\mathbf B$ is a Boolean algebra and operators $G$ and $H$ satisfy $(B1)-(B3)$. Analogously, a tense MV-algebra is a triple $(\mathbf A,G,H)$ where $\mathbf A$ is an MV-algebra and $G$ and $H$ are unary operations satisfying $(T1)-(T6)$.

Representation theorems for tense Boolean algebras (see \cite{10}) and tense MV-algebras defined on semisimple MV-algebras (see \cite{BoPa}) say that tense operators $G$ and $H$ are both $\phi$-type functions such that $G$ and $\neg H\neg$ forms Galois connection induced by a Boolean relation. We show that in the case of Pavelka's logic (and also in the case of MV-algebras) the concept of tense MV-algebras is stronger than it should be.

\begin{defin}
Let  $\mathbf A=(A;\oplus,\neg,\{\mathbf r\mid r\in [0,1]\cap \mathbb Q\})$ be a Pavelka's algebra and let $G,H\colon A\longrightarrow A$ be operators such that for every $x,y\in A$,
\begin{itemize}
\item[(PT1)] $G(x\wedge y) = G(x)\wedge G(y)$ and $H(x\wedge y) = H(x)\wedge H(y)$,
\item[(PT2)] $\mathbf r\rightarrow G(x)= G(\mathbf r\rightarrow x)$ and $\mathbf r\rightarrow H(x)= H(\mathbf r\rightarrow x)$,
\item[(PT3)] $\neg H\neg G (x)\leq x$ and $\neg G\neg H (x)\leq x$.
\end{itemize}
\end{defin}

A \emph{time frame} is a couple $(T,R)$ where $T$ is a set and $R\colon T^2\longrightarrow [0,1]$ is a fuzzy relation. A representation theorem is a direct corollary of the previously proved ones.

\begin{theorem}[Representation Theorem for tense Pavelka's algebras]\label{repretense}
i) If $(T,R)$ is a time frame then an algebra $([0,1]^T;f_R,f_{R^-1})$ is  a tense Pavelka's algebra. We call this algebra a \emph{tense Pavelka's algebra induced by the time frame $(T,R)$}.

ii) Any tense Pavelka's algebra defined on a semisimple Pavelka's algebra is embeddable into a tense Pavelka's algebra induced by a time frame. We call this embedding and time frame (described in Theorem \ref{pt1})  a \emph{natural embedding} and a \emph{natural time frame}. 
\end{theorem}
\begin{proof}
Axioms (PT1)-(PT3) show that $G\leftrightarrows \neg H\neg$. The rest is a corollary of the previously proved Theorems (especially of Theorem \ref{pt1}).
\end{proof}

The Following theorem compares concepts of tense algebras induced by a boolean time frame and tense MV-algebras induced by a fuzzy time frame.

\begin{theorem}\label{tenseboolean}
Let us have a tense Pavelka's algebra $(\mathbf A,G,H)$ defined on a semisimple Pavelka's algebra $\mathbf A=(A;\oplus,\neg,\{\mathbf r\mid r\in [0,1]\cap \mathbb Q\})$. Then the following statements are equivalent
\begin{itemize}
\item[i)] $G(x)\cdot G(y)\leq G(x\cdot y)$,
\item[ii)] $H(x)\cdot H(y)\leq H(x\cdot y)$,
\item[iii)] the tense Pavelka's algebra $(\mathbf A,G,H)$ is induced by a boolean time frame,
\item[iv)] the tense MV-algebra $(MV(\mathbf A),G,H)$, where $MV(\mathbf A)$ is the MV-reduct of $\mathbf A$, is a tense MV-algebra defined by (T1)-(T6).
\end{itemize}
\begin{proof}
Equivalence of the statements i), ii) and iii) follows from Theorem \ref{boolean}. The rest of the proof is contained in the representation theorem of tense MV-algebras (see \cite{BoPa}).
\end{proof}
\end{theorem}

\begin{theorem}\label{vlastn}
Let us have a tense Pavelka's algebra $(\mathbf A,G,H)$ defined on semisimple a Pavelka's algebra $\mathbf A=(A;\oplus,\neg,\{\mathbf r\mid r\in [0,1]\cap \mathbb Q\})$ then \begin{itemize}
\item[i)] the natural frame of $\mathbf A$ is reflexive if, and only if, $G(x)\leq x$ (or $H(x)\leq x$),
\item[ii)] the natural frame $\mathbf A$ is symmetric if, and only if, $H(x)=G(x)$,
\item[iii)] the natural frame $\mathbf A$ is transitive if, and only if, $G(x)\leq GG(x)$ or $H(x)\leq HH(x)$.
\end{itemize}
\end{theorem}
\begin{proof}
i) If the frame is reflexive then $G(x)(j)=f_R(x)(j)=\bigwedge_{i\in T}(R(i,j)\rightarrow x(i))\leq R(j,j)\rightarrow x(j)=x(j)$ for all $j\in T$. Thus $G(x)\leq x$. Conversely, if $G(x)\leq x$ then $R(i,i)=\bigwedge_{a\in A}(G(a)(i)\rightarrow a(i))=1$.

ii) If $R$ is symmetric then $G=f_R=f_{R^{-1}}=H$. Conversely, if $G=H$ then $R(i,j)=\bigwedge_{a\in A}(G(a)(j)\rightarrow a(i)) = \bigwedge_{a\in A}(H(a)(j)\rightarrow a(i))=R(j,i)$.

iii) If $R$ is transitive then
\begin{eqnarray*}
GG(x)(k) &=& \bigwedge_{j\in T}(R(j,k)\rightarrow G(x)(j))\\
&=& \bigwedge_{j\in T}\big(R(j,k)\rightarrow \big(\bigwedge_{i\in T}(R(i,j)\rightarrow x(i)\big)\big)\\
&=& \bigwedge_{i,j\in T}((R(i,j)\cdot R(j,k))\rightarrow x(i))\\
&\geq & \bigwedge_{i\in T}( R(i,k))\rightarrow x(i))\\
&=& G(x)(k).
\end{eqnarray*}
Conversely, if $G(x)\leq GG(x)$ then
\begin{eqnarray*}
R(i,j)\cdot R(j,k) &=& \big( \bigwedge_{a\in A}(G(a)(j)\rightarrow a(i))\big)\cdot \big( \bigwedge_{a\in A}(G(a)(k)\rightarrow a(j))\big) \\
&\leq & \bigwedge_{a\in A} (G(a)(j)\rightarrow a(i))\cdot (GG(a)(k)\rightarrow G(a)(j))\\
&\leq & \bigwedge_{a\in A} (GG(a)(k)\rightarrow a(i))\\
&\leq & \bigwedge_{a\in A} (G(a)(k)\rightarrow a(i))\\
&=& R(i,k).
\end{eqnarray*}
\end{proof}

\subsection{Monadic Pavelka's algebras} Monadic MV-algebras has been introduced as MV-algebras enriched by a unary operation $\exists$ satisfying ($\exists$1)-($\exists$6). G. Georgescu extends the concept of monadic MV-algebras to Pavelka's algebras adding the axiom $\mathbf r = \exists \mathbf r$ (see \cite{Geo}). In this section we propose new, wider definition of monadic Pavelka's algebras and we state a representation theorem.

\begin{defin}
Let $\mathbf A=(A;\oplus,\neg,\{\mathbf r\mid r\in [0,1]\cap \mathbb Q\})$ be a Pavelka's algebra. Then by \emph{monadic Pavelka's algebra} we mean a couple $(\mathbf A,\exists)$ where $\exists$ is a closure operator satisfying $\exists \neg \exists (x)=\neg \exists (x)$ and $\mathbf r\cdot \exists (x)=\exists (\mathbf r\cdot x)$. 
\end{defin}

Our definition of monadic Pavelka's algebras is weaker then the original one. All conditions assumed in our new definition are satisfied for all 'original monadic Pavelka's algebras`. The following two theorems make clear relations between the already established concepts.

\begin{theorem}\label{monadic}
Let us have a Pavelka's algebra $\mathbf A=(A;\oplus,\neg,\{\mathbf r\mid r\in [0,1]\cap \mathbb Q\})$ and a closure operator $\exists\colon A\longrightarrow A$. Then the following propositions are equivalent
\begin{itemize}
\item[i)] the algebra $(\mathbf A, \exists )$ is a monadic Pavelka's algebra,
\item[ii)] the algebra $(\mathbf A,\forall,\forall)$, where $\forall = \neg\exists\neg$, is a tense Pavelka's algebra.
\end{itemize} 
\end{theorem}
\begin{proof}
{\it i) $\Rightarrow$ ii)} Since $\exists$ is a closure operator, it is monotone. Moreover, the axiom $\exists \neg \exists (x)=\neg \exists (x)$ yields 
$$\exists\forall(x)=\exists \neg \exists \neg(x)\leq x\leq \neg\exists\neg \exists (x)=\forall\exists (x).$$ Thus $\forall$ and $\exists$ form a Galois connection, $\forall \leftrightarrows \exists$, and $(\mathbf A, \forall,\forall)$ is a tense Pavelka's algebra.

{\it ii) $\Rightarrow$ i)} Let us assume that $(\mathbf A, \forall,\forall)$ is a tense Pavelka's algebra. Since $\exists$ is a closure operator we have $\neg\exists (x)\leq \exists\neg\exists (x)$. Moreover, $\exists\exists (x)\leq \exists (x)$ and the adjointness property we have $\exists (x)\leq \neg\exists\neg\exists (x)$  and $\exists\neg\exists (x)\leq\neg\exists (x)$. Thus $(\mathbf A,\exists)$ is a monadic Pavelka's algebra.
\end{proof}

The previous theorem and the representation theorem for tense Pavelka's algebras yields similar representation for monadic Pavelka's algebras.

\begin{theorem}\label{monadicrepre}
i) Let $I$ be an arbitrary set and let $R\colon I^2\longrightarrow [0,1]$ be a fuzzy equivalence (reflexive, symmetric and transitive fuzzy relation). Then $([0,1]^I,g_R)$ is a monadic Pavelka's algebra. We call this algebra a \emph{monadic Pavelka's algebra induced by the fuzzy equivalence $R$}.

ii) Any Pavelka's algebra $(\mathbf A,\exists)$ defined on a semisimple Pavelka's algebra $\mathbf A=(A;\oplus,\neg,\{\mathbf r\mid r\in [0,1]\cap \mathbb Q\})$ is embeddable into a Pavelka's algebra induced by some a equivalence. 
\end{theorem}
\begin{proof}
This theorem is a direct corollary of Theorems \ref{konjug}, \ref{repretense}, \ref{vlastn} and \ref{monadic}. 
\end{proof}

Finally, we show that the `original monadic Pavelka's algebras' are just monadic Paveka's algebras induced by a boolean equivalence.

\begin{theorem}
Let $(\mathbf A,\exists)$ be a monadic Pavelka's algebra defined on a semisimple Pavelka's algebra $\mathbf A=(A;\oplus,\neg,\{\mathbf r\mid r\in [0,1]\cap \mathbb Q\})$. Then the following statements are equivalent
\begin{itemize}
\item[i)] the algebra $(\mathbf A,\exists)$ is induced by a boolean equivalence,
\item[ii)] the algebra $(\mathbf A,\exists)$ is a monadic Pavelka's algebra in original sense (see \cite{Geo}),
\item[iii)] the inequality $\forall x\cdot \forall y \leq \forall (x\cdot y)$ holds for every $x,y\in A$.
\end{itemize}
\end{theorem}
\begin{proof}
The equivalence i) $\Leftrightarrow$ iii) is a corollary of Theorems \ref{boolean}, \ref{monadic} and \ref{monadicrepre}. It is easy to show that $g_R$ induced by the relation equivalence satisfies ($\exists$1)-($\exists$6) (and it was proved in \cite{CPV}, consequently). Moreover, it satisfies
 $$\exists (\mathbf r)(j)=\bigwedge_{i\sim j} \mathbf{r}(i)= r$$ for any boolean equivalence $\sim$. Hence $\exists\mathbf r=\mathbf r$. Finally, any `original monadic Pavelka's' algebra satisfies iii) and $\mathbf r\cdot \exists (x)=\exists (\mathbf r\cdot x)$ (see \cite{Geo}). 
\end{proof}

\subsection{Formal concept operators on the standard  Pavelka's algebra} As it has been mentioned above, a formal concept analysis on the standard Pavelka's algebra is based on operators $d_I$ and $h_I$ induced by a fuzzy relation $I\colon M\times G\longrightarrow [0,1]$ where $(G,M,I)$ is called a fuzzy context. Using Theorem \ref{pt2}, a reversing Galois connection $d$, $h$  between semisimple Pavelka's algebras is induced by a fuzzy formal concept if, and only if, $d\rightleftharpoons h$ holds.

Formal concepts are pairs $(h_Id_I(x),d_I(x))$. Using an easy analogy of Lemma \ref{interclosezpet} $h_Id_I$ and $d_Ih_I$ are closure operators satisfying $\mathbf r\cdot h_Id_I(x)\leq h_Id_I(\mathbf r\cdot x)$ and $\mathbf r\cdot d_Ih_I(x)\leq d_Ih_I(\mathbf r\cdot x)$. Thus the mappings $d_I$ and $h_j$ are reversing bijections between sets of closed elements (by $h_Id_I$ and $d_Ih_I$). A formal concept analysis is an analysis of a lattice of formal concepts and thus it is just the analysis of closure operators $h_Id_I$. Its characterization is contained in the following theorem.

\begin{theorem}
Let us have a semisimple Pavelka's algebra $\mathbf A=(A;\oplus,\neg,\{\mathbf r\mid r\in [0,1]\cap \mathbb Q\})$ and a closure operator $O\colon A\longrightarrow A$. Then the following statements are equivalent:
\begin{itemize}
\item[i)] $\mathbf r\cdot O(x)\leq O(\mathbf r\cdot x)$.
\item[ii)] There exists a formal context $(\s \mathbf A,G,I)$ such that $(\mathbf A,O)$ is embeddable into $([0,1]^{\s\mathbf A},h_Id_I)$.
\end{itemize}
\end{theorem}
\begin{proof}
An implication ii) $\Rightarrow$ i) is clear. The converse one we prove analogously to the Theorem \ref{clos}.

Let us denote a set $J=\{O(a)\mid a\in A\}$ and a fuzzy relation $I\colon \s\mathbf A\times J\longrightarrow [0,1]$ defined by $R(F,O(a))=O(a)/F$. Let $x\in [0,1]^J$. By an approximate element of $x$ we mean any rational-valued vector $\mathbf r\in ([0,1]\cap\mathbb Q)^J$ satisfying $\mathbf r\leq x$. Let us denote by $A(x)$ the set of all approximate elements of $x$. If $\mathbf r\in A(x)$ then we can compute
\begin{eqnarray*}
h_R(x)/F&=& \bigwedge_{O(a)\in J}(x(O(a))\rightarrow I(F,O(a)))\\
&=& \bigwedge_{O(a)\in J}(x(O(a)) \rightarrow O(a)/F)\\
&\leq& \big(\bigwedge_{O(a)\in J}(\mathbf r(O(a))\rightarrow O(a))\big)/F
\end{eqnarray*}
and, consequently, $$h_R(x)\leq \bigwedge_{O(a)\in J}(\mathbf r (O(a))\rightarrow O(a) ).$$
Therefore 
\begin{eqnarray*}
O(h_R(x)) &\leq &O\big(\bigwedge_{O(a)\in J}(\mathbf r(O(a))\rightarrow O(a))\big)\\
&\leq & \bigwedge_{O(a)\in J}O(\mathbf r(O(a))\rightarrow O(a))\\
&\leq & \bigwedge_{O(a)\in J}(\mathbf r(O(a))\rightarrow OO(a))\\
&=& \bigwedge_{O(a)\in J}(\mathbf r(O(a)) \rightarrow O(a)).
\end{eqnarray*}
Thus 
\begin{eqnarray*}
O(h_R(x)) &\leq& \bigwedge_{r\in A(x)}\bigwedge_{O(a)\in J}(\mathbf r(O(a)) \rightarrow O(a))\\
&=& \bigwedge_{O(a)\in J}\big(\big(\bigvee_{r\in A(x)}\mathbf r(O(a))\big) \rightarrow O(a)\big)\\
&=& \bigwedge_{O(a)\in J}(x(O(a)) \rightarrow O(a))\\
&=& h_R(x)
\end{eqnarray*}
And since $O$ is a closure operator, we obtain $h_R(x)=O(h_R(x))$.

On the other side, let us have any closed set $O(x)\in J$. We define $y\in A^J$ by
\begin{center}
$y(O(b))=\left\{
\begin{array}{ll}
1 &\mbox{ if } O(b)=O(x)\\
0 &\mbox{ if } O(b)\not =O(x).
\end{array}\right. $
\end{center}
and we have
$$h_R(y)/F= \bigwedge_{O(a)\in J}(y(O(a))\rightarrow O(a)/F) =O(x)/F$$ Thus, we have proved that interiors of the operator $O$ are all $h_R(y)$ such that $y\in A^J$. With respect to Corollary \ref{c1}, operators $O$ and $h_Rd_R$ have the same systems of interiors and thus $O =h_Rd_R$.
\end{proof}

\section{Resume}

In the paper we have characterized several important operators that can be derived from fuzzy binary relations and we have applied the obtained results to define the tense and monadic Pavelka's algebras in new way. We believe that the new definitions may be a good subject for a discussion how quantifiers in fuzzy logic should be introduced.

A new open problem arises: what is the correspondence between the properties of a the groupoids of fuzzy operators that are expressible as compositions of operators $\phi_R$, $\rho_R$, $\delta_R$ and $\epsilon_R$ induced by a given fuzzy relation $R$ and the properties of the relation $R$. Some examples were described in Theorem \ref{konjug}.

All the presented theory is fully extendible for more general classes of residuated lattices (for example for non-commutative bounded integral/non-integral ones). Further, the presented theory of non-associative logics cite{Bo,Bo1} can be generalized if the techniques of non-associative computing \cite{Bot} are skillfully combined with results of this paper.
 
 \section{Acknowledgements}
 
Author is very grateful to Milan Petr\' ik and Josef P\' ocs for comments.

\end{document}